\documentclass[reqno,11pt]{amsart}
\usepackage{amsmath}
\usepackage{amssymb}
\usepackage[left=1in,right=1in]{geometry}

\usepackage{amsmath}
\usepackage{amssymb}
\usepackage{amsfonts}
\usepackage{mathtools}
\usepackage{enumerate}
\usepackage{comment}
\allowdisplaybreaks

\usepackage{palatino}
\usepackage[T1]{fontenc}
\usepackage[mathscr]{eucal}
\usepackage{dsfont}

\usepackage{fullpage}

\usepackage{acronym}
\usepackage{latexsym}
\usepackage{paralist}
\usepackage{wasysym}
\usepackage{xspace}

\usepackage[dvipsnames,svgnames]{xcolor}
\colorlet{MyBlue}{DodgerBlue!60!Black}
\colorlet{MyGreen}{DarkGreen!60!Black}

\definecolor{darkred}{rgb}{1,0,0}
\definecolor{darkolivegreen}{rgb}{0.33, 0.42, 0.18}
\definecolor{darkgray}{rgb}{0.66, 0.66, 0.66}
\newcommand {\rosso}[1] {{\color{black}{#1}}}
\newcommand {\viola}[1] {{\color{black}{#1}}}

\usepackage[font=small,labelfont=bf]{caption}
\usepackage{subfigure}
\usepackage{tikz}
\usetikzlibrary{calc,patterns}

\usepackage{hyperref}
\hypersetup{
colorlinks=true,
linktocpage=true,
pdfstartview=FitH,
breaklinks=true,
pdfpagemode=UseNone,
pageanchor=true,
pdfpagemode=UseOutlines,
plainpages=false,
bookmarksnumbered,
bookmarksopen=false,
bookmarksopenlevel=1,
hypertexnames=true,
pdfhighlight=/O,
urlcolor=MyBlue!60!black,linkcolor=MyBlue!70!black,citecolor=DarkGreen!70!black, % <--- for screen
pdftitle={},
pdfauthor={},
pdfsubject={},
pdfkeywords={},
pdfcreator={pdfLaTeX},
pdfproducer={LaTeX with hyperref}
}

\numberwithin{equation}{section}  %numberwithin goes before cleverefs when using hyperref
\usepackage[sort&compress,capitalize,nameinlink]{cleveref}
\crefname{app}{Appendix}{Appendices}

\crefrangeformat{equation}{\upshape(#3#1#4)\textendash(#5#2#6)}

\usepackage[textwidth=30mm]{todonotes}

\usepackage{soul}
\setstcolor{red}
\sethlcolor{SkyBlue}

\theoremstyle{plain}
\newtheorem{theorem}{Theorem}
\newtheorem{corollary}[theorem]{Corollary}
\newtheorem*{corollary*}{Corollary}
\newtheorem{lemma}[theorem]{Lemma}
\newtheorem{proposition}[theorem]{Proposition}

\theoremstyle{definition}

\newtheorem*{definition*}{Definition}

\newtheorem*{hypothesis*}{Hypothesis}

\theoremstyle{remark}
\newtheorem{remark}[theorem]{Remark}
\newtheorem*{remark*}{Remark}
\newtheorem*{notation*}{Notational remark}

\numberwithin{theorem}{section}
\DeclarePairedDelimiterX{\braket}[2]{\langle}{\rangle}{#1,#2}

\DeclarePairedDelimiterX{\inner}[2]{\langle}{\rangle}{#1,#2}
\DeclarePairedDelimiterX{\setdef}[2]{\{}{\}}{#1:#2}

\newcommand{\ind}{\mathds{1}}

\newcommand{\cX}{\ensuremath{\mathcal X}}

\newcommand{\E}{\ensuremath{\mathbb{E}}}

\newcommand{\C}{\ensuremath{\mathbb{C}}}
\newcommand{\N}{\ensuremath{\mathbb{N}}}

\newcommand{\R}{\ensuremath{\mathbb{R}}}
\renewcommand{\P}{\ensuremath{\mathbb{P}}}
\newcommand{\pl}{\ensuremath{\left\langle}}
\newcommand{\pr}{\ensuremath{\right\rangle}}

\def\({\left(}
\def\){\right)}
\def\[{\left[}
\def\]{\right]}
\newacro{NE}{Nash equilibrium}
\newacroplural{NE}[NE]{Nash equilibria}
\newacro{PNE}{pure Nash equilibrium}
\newacroplural{PNE}[PNE]{pure Nash equilibria}
\newacro{MNE}{mixed Nash equilibrium}
\newacroplural{MNE}[MNE]{mixed Nash equilibria}
\newacro{PFNE}{prior-free Nash equilibrium}
\newacroplural{PFNE}[PFNE]{prior-free Nash equilibria}
\newacro{WE}{Wardrop equilibrium}
\newacroplural{WE}[WE]{Wardrop equilibria}
\newacro{SO}{socially optimum}
\newacro{SU}{social utility}
\newacro{BEq}{best equilibrium}
\newacro{WEq}{worst equilibrium}
\newacro{KKT}{Karush\textendash Kuhn\textendash Tucker}
\newacro{OD}[O/D]{origin-destination}
\newacro{PoA}{price of anarchy}
\newacro{PoS}{price of stability}
\newacro{PoCS}{price of correlated stability}
\newacro{BPR}{bureau of public roads}
\newacro{FIP}{finite improvement property}
\newacro{CLT}{central limit theorem}

\newacro{BPG}{buck-passing game}
\newacro{SBPG}{stochastic buck-passing game}
\newacro{MBPG}{mixed extension of the buck-passing game}

\begin{document}
\title[A probabilistic proof of Cooper\&Frieze's {\em First Visit Time Lemma}]{A probabilistic proof of Cooper\&Frieze's \\``First Visit Time Lemma''}
\author[F.~Manzo]{Francesco Manzo$^{*}$}
\address{$^{*}$ Dipartimento di Matematica e Fisca, Universit\`a di Roma Tre, Largo San Leonardo Murialdo 1, 00146 Roma, Italy.}
\email{manzo.fra@gmail.com}
\author[M.~Quattropani]{Matteo Quattropani$^{\dagger}$}
\address{$^{\dagger}$ Dipartimento di Economia e Finanza, LUISS, Viale Romania 32, 00197 Roma, Italy.}
%\email{matteo.quattropani@uniroma3.it}
\email{mquattropani@luiss.it}
\author[E.~Scoppola]{Elisabetta Scoppola$^{\#}$}
\address{$^{\#}$ Dipartimento di Matematica e Fisca, Universit\`a di Roma Tre, Largo San Leonardo Murialdo 1, 00146 Roma, Italy.}
\email{scoppola@mat.uniroma3.it}
%\urladdr{\url{http://www.there.com}}
\maketitle              % typeset the header of the contribution

\begin{abstract} In this short note we present an alternative proof of the so-called \emph{First Visit Time Lemma} (FVTL), originally presented by Cooper and Frieze in its first formulation in \cite{CF1}, and then used and refined in a list of papers by Cooper, Frieze and coauthors. We work in the original setting, considering a growing sequence of irreducible Markov chains on $n$ states. We assume that the chain is rapidly mixing and with a stationary measure having no entry which is too small nor too large. Under these assumptions, the FVTL shows the exponential decay of the distribution of the hitting time of a given state $x$---for the chain started at stationarity---up to a small multiplicative correction. While the proof of the FVTL presented by Cooper and Frieze is based on tools from complex analysis, 
and it requires an additional assumption on a generating function, we present a completely probabilistic proof, relying on the theory of quasi-stationary distributions and on strong-stationary times arguments. In addition, under the same set of assumptions, we provide some quantitative control on the Doob's transform of the chain on the complement of the state $x$.
\end{abstract}

\section{Introduction}\label{sec:intro}
In the early 00's, Cooper and Frieze started a series of papers on which they compute the first order asymptotics of the cover time of random walks on different random graphs, see \cite{CF2,CF3,CF4,CF5,CF6,CF8,CF9}. Given an arbitrary (possibly directed) graph structure, the cover time is the expected time needed by a simple random walk to visit every vertex of the graph, maximized over all the possible starting positions. 
One of the key ingredients of Cooper and Freze's analysis is the so called \emph{First Visit Time Lemma (FVTL)}, as named by the authors in \cite{CF1}.
The same lemma has been of use in proving also different kind of results, e.g., to estimate expected meeting time of multiple random walks on random graphs, see \cite{CFRmultiple}. The lemma deals with the tail probability of the stopping time $\tau_x$, i.e., the time of the first visit to the state $x$. Consider a sequence of Markov chains on a growing state space of size $n$. We assume that for every sufficiently large $n$ the chain is irreducible, admitting a unique invariant measure $\pi=\pi_n$. The framework of the lemma is based on two additional crucial assumptions relating mixing time and spread of the stationary measure, namely, we assume the existence of a time $T=T_n$ such that
\begin{equation}\label{hp1}
\max_{x,y}\left|P^T(x,y)-\pi(y) \right|=O \(\frac{1}{n^3}\),
\end{equation}
and
\begin{equation}\label{hp2}
T\:\max_{x}\pi(x)=o(1),\qquad \min_x\pi(x)=\omega(n^{-2}).
\end{equation}

Under the latter assumptions and 
adding a technical requirement on the generating function of the recurrences to a fixed state $x$, 
the authors show that \rosso{starting from any state} $y$ and for all $t>T$:
\begin{equation}
\qquad \P_y\(\text{the process does not visit $x$ in the interval $[T,t]$ }\)\sim \(1-\frac{\pi(x)}{R_T(x)}\)^{t},
\end{equation}
where $R_T(x)\ge 1$ is the expected number of returns in $x$ within the mixing time $T$.
The proof of the latter results, as well as the underlying technical assumptions, evolved with their uses since the first formulation in \cite{CF1} to the last (to the best of our knowledge) formulation and proof in \cite{CF4}. 
\textcolor{black}{We remark that the assumptions in \cref{hp1,hp2} are typically satisfied by random walks on many models of random graphs, e.g., Erd\H{o}s-Renyi graphs or configuration models.}

The techniques used in the proof by Cooper and Frieze rely on probability arguments but also on tools from complex analysis and an analytical expansion of some probability generating functions.
In this paper we aim at finding a probabilistic proof of the FVTL, trying to shed some light on the underlying phenomenology. On the technical side, the arguments in our proof are elementary and 
do not need the additional assumption on the generating function required
in the original Cooper and Frieze's proof. We refer to \cref{sec:CF} for a direct comparison of our result with the original one.

\begin{color}{black}
Exponential law of hitting times is a classic and widely studied topic in probability. We just recall here  the pioneering book
by Keilson  \cite{keilson} and the beautiful papers by Aldous (see \cite{A82} and also \cite{AB1, AB2}). In \cite{A82}, Aldous recognizes two regimes in which the latter phenomenon takes place:
\end{color}
\begin{enumerate}
	\item A single state $m$ is frequently visited before $\tau_{x}$. When
	starting from $m$, the path to $x$ consists of a geometric number
	of excursions (with mean $\left(\P_{m}(\tau_{m}>\tau_{x}\right)^{-1}$)
	from $m$ to $m$ without touching $x$, before the final journey
	to $x$. 	
	The hitting time is dominated by the sum of many i.i.d. excursion times and therefore it is almost exponential \cite{keilson}. 
	\item When the chain is rapidly mixing, then the distribution
	 \viola{at time $t$} is near to the stationary distribution \viola{even when conditioned on $\tau_x>t$}.  This case is analyzed in \cite{A82}, where Aldous shows that
	$$\sup_{t\ge 0}\left|\P_\pi(\tau_x>t)-e^{-\frac{t}{\E_\pi[\tau_x]}} \right|\le \delta,$$
	where $\delta$ is a function of the mixing time of the chain and of the expectation $\E_\pi[\tau_x]$. Aldous shows that, if the hitting of $x$ is a rare event, i.e., the expectation of $\tau_x$ is much larger \rosso{than}  the mixing time of the chain, then $\delta$ is small. 
	\end{enumerate}

\begin{color}{black}
In the early years, these two regimes were considered as complementary. One of the main applications of the scenario in (1) has been the study of metastability, namely the behavior of processes that are trapped for a long time in a part of their state space. Before exiting the trap, the process visits many times a ``metastable state'', reaching an apparent, local equilibrium. In such systems the exit from the trap triggers the relaxation to equilibrium so that relaxation to equilibrium can be discussed as the first hitting to the complement of the trap. We refer to \cite{OV,BdH} for a general introduction to metastability and to \cite{BG,BGM,FMNS,FMNSS, MS} for a discussion of the extension of metastability methods to other regimes. 

The FVTL frames in scenario (2) and it was proved by means of a different set of techniques.  Aldous' result mentioned in (2)  has an additive error term and therefore it cannot provide first-order asymptotics of the exponential approximation when $t$ is large, in contrast to the FVTL where a multiplicative bound is proved.

More recently, these two regimes begin to be understood in a common framework, by generalizing recurrence ideas to measures instead of 
recurrence to points. The quasi-stationary measure, introduced in the pioneering paper by Darroch and Seneta \cite{DS} (see also  \cite{CMSM}, and \cite{Pbib} for a more recent bibliography on the subject), plays the role of a recurrent measure before the hitting. The hitting to the measure can be studied by extending the theory of strong stationary times
\cite{AD1, AD2, LevPer:AMS2017}, to quasi-stationarity, see \cite{DMqs,MS}. In particular, the notion of \emph{conditional strong quasi-stationary time} introduced in \cite{MS}, has shown to be useful in providing exact formulas for the distribution of the first hitting time $\tau_x$ starting from an arbitrary distribution.
An introduction to these tools is given in the following subsection where a rough estimate on the tail of $\tau_x$ is given. Under the strong hypotheses considered in this paper we can follow an easier way, involving the quasi-stationary measure but not requiring the use of conditional strong quasi-stationary times. Indeed, in our case the stationary measure and the quasi-stationary one are very close to each other. The more general results obtained in \cite{MS} could be useful in considering more general regimes with different starting measure.
The final part of this paper is devoted to the discussion of our proof in this perspective.
\end{color}

\subsection{A first discussion}
For any $x\in\cX$,  let $\tau_x$ denote the hitting time of $x$, namely
\begin{equation}\label{eq:hitting}
\tau_x=\inf\{t\ge 0\:|\:X_t=x \}.
\end{equation}
We will call $[P]_x$ the sub-Markovian probability kernel obtained by removing the $x$-th row and column by the matrix $P$. We will assume that $[P]_x$ is a primitive sub-Markovian kernel, i.e.,  all entries of $([P]_x)^m$
are positive for some $m\in\mathbb N$. By the Perron-Froboenius theory (see, e.g., \cite{CMSM}) there exists a unique probability distribution $\mu^\star_x$ and a real $\lambda_x<1$
\begin{equation}\label{quasi0}
\mu^\star_x[P]_x =\lambda_x \mu^\star_x ,
\end{equation}
Moreover, we denote by $\gamma_x$ the corresponding right eigenvector, i.e.,
\begin{equation}\label{gamma}
[P]_x\gamma_x=\lambda_x\gamma_x,
\end{equation}
normalized by
$
\pl\gamma_x,\mu_x^\star \pr=1$.

The probability distribution $\mu^\star_x$ is called {\it quasi-stationary measure} and it is strictly related to the exponential behavior of the tail probability $\P(\tau_x>t)$.
Indeed, when looking at the evolution of the process starting from $\mu^\star_x$, by \cref{quasi0} we deduce 
\begin{equation}\label{quasiexp}
\P_{\mu^\star_x}(\tau_x>t)=\sum_z \mu^\star_x(z)\P_z(\tau_x>t)=\sum_{z\not=x} {\mu^\star_x(z) \sum_{y\not=x}\big([P]_x\big)^t(z,y)}=\lambda_x^t\sum_{y\not=x}
\mu^\star_x(y)=\lambda_x^t.
\end{equation}
For more details see \cite{DMqs,DMqs2,Mabsorption,Mmetastability}, the application to the metastability regime are discussed
in \cite{FMNS,FMNSS,MS}.

The  right eigenvector $\gamma_x$ defined in \cref{gamma} controls the dependence on
the initial distribution of the probability of the event $\tau_x>t$.
Indeed this eigenvector is related to the asymptotic ratios
of the right tail probabilities, see \cite[Eq. (3.5)]{CMSM}
\begin{equation}\label{eq:gamma-ratio}
\lim_{t\to \infty}\frac{\P_y(\tau_x>t)}{\P_z(\tau_x>t)}=\frac{\gamma_x(y)}{\gamma_x(z)}\qquad y,z\not= x.
\end{equation}
With this right eigenvector we can construct a \emph{Local Chain} on $\cX\setminus\{x\}$,  which is usually referred to as \emph{Doob's transform of $X$}.
For any $y,z\not= x$, define the stochastic matrix
\begin{equation}\label{tildeP}
\widetilde P(z,y):=\frac{\gamma_x(y)}{\gamma_x(z)} \frac{P(z,y)}{\lambda_x}.
\end{equation}
More generally 
\begin{equation}\label{tildePt}
\widetilde P^t(z,y)=\frac{\gamma_x(y)}{\gamma_x(z)} \frac{\big([P]_x\big)^t(z,y)}{\lambda_x^t}\qquad\forall t\ge 0.
\end{equation}
It is immediate to show that $\widetilde P$ is a primitive matrix. The invariant measure of the latter chain is 
$$\nu(y):=\gamma_x(y)\mu^\star_x(y).$$ 
For the chain $\widetilde X$ we define
\begin{equation}
\tilde s^z(t,y):=1-\frac{\widetilde P^t(z,y)}{\nu(y)}
\end{equation}
and will call separation distance at time $t$ the quantity $\tilde s(t)$ defined as
\begin{equation}
\tilde s(t):=\sup_{z\not=x}\tilde s^z(t)\qquad\text{where}\qquad\tilde s^z(t):=\sup_{y\not= x}\tilde s^z(t,y).
\end{equation}
Note that $\tilde s^z(t)\in[0,1]$ and recall that $\tilde s(t)$ has the sub-multiplicative property
$$\tilde s(t+u)\le \tilde s(t)\tilde s(u),$$
which in particular implies an exponential decay in time of $\tilde s$, see \cite{LevPer:AMS2017}.

Consider any initial measure $\alpha$ on $\cX\setminus\{x\}$ and define the transformation
\begin{equation}\label{starting}
\tilde \alpha(y) := \frac{\alpha(y)\gamma_x(y)}{\pl\alpha,\gamma_x\pr },	\qquad\forall y\neq x.
\end{equation}
Then, as shown in \cite{MS}, 
\begin{align}
\P_\alpha(\tau_x>t)=&\sum_{y\neq x}\sum_{z\neq x}\alpha(z) {\big([P]_x\big)^t(z,y)}\\
=&\sum_{y\not=x}\sum_{z\not=x}\alpha(z) {\gamma_x(z)\lambda_x^t\mu_x^\star(y)\frac{\widetilde P^t(z,y)}{\nu(y)}}\\
=&\lambda_x^t\sum_{z\not=x}\alpha(z) \gamma_x(z)\sum_{y\not=x}\mu^\star_x(y)(1-\tilde{s}^z(t,y))\\
\label{ubrozzo}=&\lambda_x^t\pl \alpha,\gamma_x\pr\Big(1-\sum_{y\not=x}\mu^\star_x(y)\tilde s^{\tilde\alpha}(t,y)\Big)
\end{align}
where we call
\begin{equation}\label{stildeatilde}
\tilde s^{\tilde\alpha}(t,y) := \sum_{x\not= x}\tilde\alpha(x)\tilde s^x(t,y)\qquad \hbox{ and} \qquad \tilde s^{\tilde\alpha}(t):=\sup_{y\not= x}\tilde s^{\tilde\alpha}(t,y).
\end{equation}
Moreover, again by \cite{MS}, we know that \cref{ubrozzo} can be estimated from above and below by
\begin{equation}\label{eq:bounds}
\lambda_x^t\pl \alpha,\gamma_x\pr \Big(1-\tilde s^{\tilde\alpha}(t)\Big)\le \P_\alpha(\tau_x>t)\le \lambda_x^t\pl \alpha,\gamma_x\pr \(1+\tilde s^{\tilde\alpha}(t)\(\frac{1}{\min_y\gamma_x(y)}-1\)\).
\end{equation}
\cref{eq:bounds} suggests that, in the regime in which $|\cX|\to\infty$, the first order geometric approximation of the tail probability $\P_\alpha(\tau_x>t)$ can be obtained. In particular, the exponentiality immediately follows from \cref{eq:bounds} for all those Markov chains $P$, target states $x$, initial distributions $\alpha$ and time $t$ for which all of the following assumptions hold:
\begin{enumerate}[(i)]
	\item\label{it:i} $s^{\tilde\alpha}(t)=o(1)$, i.e., $t$ is sufficiently large to have that the Doob transform starting at $\alpha$ is well mixed by time $t$;
	\item\label{it:ii} $\pl \alpha,\gamma_x\pr \sim 1$, which occurs in particular if $\gamma_x$ approximates the constant vector;
	\item\label{it:iii} $\min_y\gamma_x(y)=\Omega(1)$, which can be thought of as an additional uniformity requirement to the one in \cref{it:ii}.
\end{enumerate}

Despite the intuitions based on \cref{eq:bounds}, we are not going to follow exactly the heuristic recipe explained in \cref{it:i,it:ii,it:iii}. In fact our focus is on the special case in which $\alpha=\pi$, which leaded us through a different path toward proving exponentiality. Nevertheless, as a byproduct of our proof of the FVTL we provide uniform upper and lower bounds on the right eigenvector $\gamma_x$. We think those bounds can be of independent interest, since they can be turned into a quantitative information on the structure of the Doob's transform of the process $X$. In particular, for a given model, our bounds could be useful in verifying the conditions in \cref{it:i,it:ii,it:iii}, and therefore in finding---for every fixed choice of the initial distribution $\alpha$---the right first order approximation of the decay of $\P_\alpha(\tau_x>t)$.

\section{Notation and results}
We start by presenting the notation and briefly recalling the basic quantities introduced in \cref{sec:intro}. We consider a sequence of Markov chains on a growing state space. Formally:
\begin{itemize}
	\item $\cX^{(n)}$ is a state space of size $n$.
	\item $(X^{(n)})_{t\ge 0}$ is a discrete time Markov chain on $\cX^{(n)}$.
	\item $\P^{(n)}$ is the probability law of the Markov chain $(X^{(n)})_{t\ge 0}$, and $\E^{(n)}$ the corresponding expectation.
	\item $P^{(n)}$ is the transition matrix of $(X^{(n)})_{t\ge 0}$, which is assumed to be ergodic.
	\item $\pi^{(n)}$ is the stationary distribution of $P^{(n)}$.
	\item For any probability distribution $\alpha$ on $\cX^{(n)}$ and every integer $t\ge 0$, we note by $\mu^\alpha_t$ the probability distribution of the chain $X^{(n)}$ starting at $\alpha$ and evolved for $t$ steps, i.e., 
	$$\mu^{\alpha}_t(y):=\sum_{x\in\cX^{(n)}}\alpha(x)\big(P^{(n)}\big)^t(x,y),\qquad \forall y\in\cX^{(n)}.$$
	\item For all $x\in\cX^{(n)}$, $\tau_{x}$ represents the hitting time of vertex $x$, defined as in \cref{eq:hitting}.
	\item For all $t\ge 0$ and $x\in\cX^{(n)}$, we let the symbol $\zeta_t(x)$ denote the random time spent by the process $X^{(n)}$ in the state $x$ within time $t$, i.e.,
	\begin{equation}\label{eq:local-time}
	\zeta_t(x):=\sum_{s=0}^{t-1} \ind_{X^{(n)}_s=x}.
	\end{equation}
	\item For all $x\in\cX^{(n)}$ we denote by $[P^{(n)}]_{x}$ the sub-Markovian kernel  obtained by removing the $x$-th row and column of $P^{(n)}$. The kernel $[P^{(n)}]_{x}$ is assumed to be irreducible.
	\item  For all $x\in\cX^{(n)}$, $\lambda_{x}$ denotes as the leading eigenvalue of $[P^{(n)}]_{x}$ and $\mu^\star_{x}$ as the corresponding left eigenvector, normalized so that $\mu^\star_{x}$ is a probability distribution over $\cX^{(n)}\setminus\{x \}$. See \cref{quasi0}. We remark that by the definitions follows that
	\begin{equation}
	\P_{\mu^\star_{x}}(\tau_{x}>t)=\lambda_{x}^t,\qquad\forall t\ge 0,
	\end{equation}
	see \cref{quasiexp}.
	\item For all $x\in\cX^{(n)}$, $\gamma_{x}$ denotes the right eigenvector of $[P^{(n)}]_{x}$ associated to the eigenvalue $\lambda_{x}$. We consider $\gamma_{x}$ to be normalized so that $\pl \mu_{x}^\star,\gamma_{x}\pr =1$.
\end{itemize}
 Since we are interested in asymptotic results when $n\to\infty$, the asymptotic notation will refer to this limit and the explicit dependence on $n$ will be usually dropped.

We will adopt the usual asymptotic notation $(o,O,\Theta,\omega,\Omega)$ and, given two functions $f,g:\N\to\R_+$, we will use the symbols $\sim$ and $\lesssim$ with the meaning
$$f(n)\sim g(n)\qquad \iff\qquad \lim_{n\to\infty}\frac{f(n)}{g(n)}=1,$$
and
$$f(n)\lesssim g(n)\qquad\iff\qquad  \limsup_{n\to\infty}\frac{f(n)}{g(n)}\le 1,$$
respectively.
\subsection{Results}\label{sec:results}
We will work under the following asymptotic assumption for the sequence of Markov chains: There exist
\begin{itemize}
	\item A real number $c>2$.
	\item A diverging sequence $T=T(n)$ .
\end{itemize}
such that
\begin{itemize}
	\item[\textbf{(HP 1)}]\label{hp:mix} Fast mixing: 
	$$\max_{x,y\in\cX}\left|\mu_T^x(y)-\pi(y) \right|=o(n^{-c}).$$
	\item[\textbf{(HP 2)}]\label{hp:pimax} Small $\pi_{\max}$:
	 $$ T\max_{x\in\cX} \pi(x)=o(1).$$
	\item[\textbf{(HP 3)}]\label{hp:pimin} Large $\pi_{\min}$:
	$$\min_{x\in\cX}\pi(x)=\omega(n^{-2}).$$
\end{itemize}
Fixed any $x\in\cX$ we let $R_T(x)$ denote the expected number of returns at $x$ for the Markov chain starting at $x$ within $T$. More precisely,
\begin{equation}
R_T(x)=\sum_{t=0}^T \mu_t^x(x)\ge 1.
\end{equation}
The precise statement that we prove is the following
\begin{theorem}[First Visit Time Lemma]\label{fvtl}
Under the assumptions \textbf{(HP1)}, \textbf{(HP2)} and \textbf{(HP3)} for all $x\in\cX$, it holds
\begin{equation}\label{fvtleq}
\sup_{t\ge 0}\left| \frac{\P_{\pi}(\tau_x>t)}{\lambda_x^t} - 1 \right|\longrightarrow 0,
\end{equation}
and
\begin{equation}
\left| \frac{\lambda_x}{\(1-\frac{\pi(x)}{R_T(x)}\)} - 1 \right|\longrightarrow 0.
\end{equation}
\end{theorem}
We will see in \cref{sec:doob} that it follows as an easy consequence of \cref{fvtl} that the right-eigenvector $\gamma_x$ asymptotically has mean 1 with respect to the stationary distribution. In other words, the following corollary holds.
\begin{corollary}\label{co:pidotgamma}
Under the same assumptions of \cref{fvtl}: for all $x\in\cX$
\begin{equation}
\sum_{y\in\cX\setminus\{ x\}}\pi(y)\gamma_x(y)\to 1.
\end{equation}
\end{corollary}
Moreover, we provide some entry-wise upper and lower bound for the eigenvector $\gamma_x$. 

\begin{theorem}\label{th:gamma}
	Under the same set of assumptions, for every $x\in\cX$:
	\begin{enumerate}[(i)]
		\item For all $y\in\cX\setminus\{x\}$ it holds $$\gamma_x(y)\lesssim1.$$
		\item For all $y\in\cX\setminus\{x\}$ it holds
		$$\gamma_x(y)\gtrsim \big[1-\E_y\[\zeta_T(x) \]\big]_{+}.$$
	\end{enumerate}
\end{theorem}
\begin{remark}
We remark that the asymptotic lower bound in \cref{th:gamma} is in fact not void for most of the models of random graphs which are known to satisfy the assumptions of the FVTL. As an example, if $X$ is the simple random walk on a random regular directed graph of in/out-degree $r$, then---with high probability with respect to the construction of the environment---for every $x\in\cX$ the quantity $\E_y[\zeta_T(x)]$ is strictly smaller than $1$ uniformly in $y\in\cX\setminus\{x\}$; moreover, $\E_y[\zeta_T(x)]=0$ for most $y\in\cX\setminus\{x\}$. To see the validity of the latter statement, we refer the reader to \cite[Propositions 4.3 and 4.4]{CQcover}.
\end{remark}

\begin{color}{black}
\subsection{Comparison with Cooper\&Frieze's lemma}\label{sec:CF}
In order to facilitate a direct comparison, we write here---using our notation---the claim proved by Cooper and Frieze, stressing the differences with \cref{fvtl}.
\begin{theorem}[See Lemma 6 and Corollary 7 in \cite{CF4}.]\label{fvtl-cf}
	Consider a sequence of Markov chains satisfying the assumptions \textbf{(HP1)}, \textbf{(HP2)} and \textbf{(HP3)} with $c=3$. Moreover, let
	$$a=\frac{1}{KT} $$
	for a suitably large constant $K$. Fix $x\in\cX$ and assume further that the truncated probability generating function
	$$\mathbf{R}(z)=\sum_{t=0}^{T-1}P^t(x,x)z^t,\qquad\forall z\in \C$$
	satisfies
	\begin{equation}\label{hpcf}
	\min_{|z|\le 1+a}\mathbf{R}(z)\ge\theta
	\end{equation}
	for some constant $\theta>0$. Then, for all $y\in\cX$ and $t\ge 0$
	\begin{equation}\label{fvtlcf}
	\P_{\mu^T_y}\(\tau_x>t \)=\big(1+O(T\pi(x)) \big)\tilde{\lambda}_x^{t}+o\(e^{-a t/2}\),
	\end{equation}
	where
	$$\tilde\lambda_x=\(1+\frac{\pi(x)}{R_T(x)(1+O\(T\pi(x) \))}\)^{-1}.$$
\end{theorem}
Even at a first sight, there are three main differences between \cref{fvtl-cf} and \cref{fvtl}:
\begin{enumerate}
\item First, our proof neglects the technical assumption in \cref{hpcf}. Indeed, we remark once again are not going to use any tool complex analysis, being our proof elementary and completely probabilistic in nature.
\item Second, the estimate in \cref{fvtlcf} concerns the tail probability of the hitting time when the initial measure is the $T$-step evolution starting at any fixed vertex $y$. The latter is in fact a minor difference. In \cref{le:approx} we will show that our estimate in \cref{fvtleq} holds even when replacing $\pi$ by $\mu_y^T$, for any choice $y$.
\item Finally, our result does not take into account the precise magnitude of the second order corrections. This is because we would like to put the accent of this paper on the underlying phenomenology, trying to keep the paper as easy and readable as possible. We stress that more precise bounds could be obtained through the same set of arguments.
\end{enumerate}
\subsection{Overview of the paper}\label{sec:overview}
\cref{sec:proof} is devoted to the proof of \cref{fvtl}. The proof is divided into several steps. We start by showing a first order approximation for the expected hitting time of $x$ starting at stationarity, i.e. $\E_\pi[\tau_x]\sim R_T(x)/\pi(x)$. See \cref{pr:abd}. In order to show that the latter expectation coincides at first order with $\E_{\mu^\star_x}[\tau_x]$ we prove that the tail probability $\P_{\pi}(\tau_x>t)$ is asymptotically larger or equal to the the tail of the same probability starting at any other measure. This is the content of 
\cref{le:max}. To conclude the validity of 
\begin{equation}\label{sharper}
\E_\pi[\tau_x]\sim \E_{\mu^\star}[\tau_x]=(1-\lambda_x)^{-1},
\end{equation}
we then use a bootstrap argument: we first show in \cref{le:lambda-small} that $\lambda_x^T\sim 1$, then---in \cref{pr:mean-pi}---we show that the latter bound can be translated in the sharper estimate in \cref{sharper}. Once established  \cref{sharper}, the exponential approximation can be obtained by using the properties of quasi-stationary distributions.

In \cref{sec:doob} we use the understanding developed in \cref{sec:proof} to show the validity of \cref{co:pidotgamma,th:gamma}. Namely, we see how the FVTL reflects on the properties of the first right-eigenvector $\gamma_x$.

Finally, in \cref{sec:conditional}, we aim at framing the FVTL and its setting in the language of  \emph{conditional strong quasi-stationary times} introduced in \cite{MS}. 
 \end{color}

\section{Proof of the FVTL}\label{sec:proof}
As mentioned in \cref{sec:overview}, our proof of the \cref{fvtl} is divided into several small steps. The first proposition is devoted to the computation of the average hitting time of $x$ starting at stationarity. The credits for this result go to Abdullah, who presented it in his PhD thesis, \cite[Lemma 58]{MAphd}. We repeat here the proof for the reader's convenience.
\begin{proposition}[see \cite{MAphd}]\label{pr:abd}
	For all $x\in\cX$
	\begin{equation}
	\E_\pi[\tau_x]\sim \frac{R_T(x)}{\pi(x)}.
	\end{equation}
\end{proposition}
\begin{proof}
By \cite[Lemma 2.1]{AF} we have
$$\E_\pi[\tau_x]=\frac{Z(x,x)}{\pi(x)},$$
where $Z$ is the so called \emph{fundamental matrix}, defined by
\begin{equation}
Z(x,x):=\sum_{t=0}^{\infty}\mu_t^x(x)-\pi(x).
\end{equation}
By the submultiplicativity of the sequence
\begin{equation}
D(t):=\max_{x,y}\left|\mu_t^x(y)-\pi(y) \right|,
\end{equation}
i.e.,
\begin{equation}
D(t+s)\le 2D(t)D(s),\qquad\forall t,s>0,
\end{equation}
and thanks to \textbf{(HP 1)}, we have
\begin{equation}
\max_{x,y}\left|\mu_{kT}^x(y)-\pi(y) \right|\le \left( \frac{2}{n^c}\right)^k,\qquad\forall k\in\N.
\end{equation}
Hence,
\begin{align*}
Z(x,x)=&\sum_{t\le T}\big(\mu_T^x(x)-\pi(x) \big)+T\sum_{k\ge 1} \left( \frac{2}{n^c}\right)^k\\
=& R_T(x)+O(T\pi(x))+O(Tn^{-c})\\
=&R_T(x)(1+o(1)),
\end{align*}
where in the latter asymptotic equality we used  $Tn^{-c}\le T\pi_{\max}$, \textbf{(HP 2)}, and the fact that $R_T(x)\ge 1$.
\end{proof}
%\MQ{We could remove the proof above if we lack some space.}
\begin{remark}
We remark that, by the \emph{eigentime identity} (see \cite{AF,PTtree,Meigentime}) the trace of the fundamental matrix of an irreducible chain coincides with the sum of the inverse non-null eigenvalues of the generator, which in turn coincide with the  expected hitting time of a state sampled accordingly to the stationary distribution. Namely, for all $y\in\cX$,
\begin{equation}
\sum_{x\in\cX}\pi(x)\E_y\tau_x=\sum_{x\in\cX}Z(x,x)=\sum_{i=2}^n \frac{1}{1-\theta_i}
\end{equation}
where
$$1=\theta_1>\Re (\theta_2)\ge\dots\ge \Re(\theta_n)\ge-1$$
are the eigenvalues of $P$. By \cref{pr:abd} we get that, for all $y\in\cX$,
\begin{equation}
\sum_{x\in\cX}Z(x,x)\sim \sum_{x\in\cX}R_T(x).
\end{equation}
In other words, under the assumptions in \textbf{(HP 1)}, \textbf{(HP 2)} and \textbf{(HP 3)}, the sum of the inverse eigenvalues of the generator can be well approximated by the sum of the expected returns within the mixing time.
\end{remark}
A crucial fact that will be used repeatedly in what follows is that under the assumptions in \cref{sec:results}, the tails of $\tau_x$ starting at $\mu_T^y$ and starting at $\pi$ coincide at first order.
\begin{lemma}\label{le:approx}
	For all $x,y\in\cX$ and $t>0$ it holds
	\begin{equation}
		\P_{\mu_T^y}(\tau_x>t)\sim\P_\pi(\tau_x>t).
	\end{equation}
\end{lemma}
	\begin{proof}
		By the assumptions we have that 
		\begin{align}
		\label{eq:trick1}	\max_{x,y\in\cX}\left|\frac{\mu_T^x(y)}{\pi(y)}-1 \right|=&\max_{x,y\in\cX}\frac{1}{\pi(y)}\left|\mu_T^x(y)-\pi(y) \right|\\
		\label{eq:trick2}	\le&\frac{1}{\min_{y\in\cX}\pi(y)}\max_{x,y\in\cX}\left|\mu_T^x(y)-\pi(y) \right|\\
		\label{eq:trick3}	\text{By \textbf{(HP 1)}}\Longrightarrow\quad\le &\frac{n^{-c}}{\min_{y\in\cX}\pi(y)}\\
		\label{eq:trick4}	\text{By \textbf{(HP 3)}}\Longrightarrow\quad=&o(n^{-c+2}),
		\end{align}
		from which the claim follows. In fact,
		\begin{equation*}
		\P_{\mu_y^T}(\tau_x>t)=\sum_{z}\mu_y^T(z)\P_x(\tau_x>t)= (1+o(1))\sum_{z}\pi(z)\P_x(\tau_x>t)=(1+o(1))\P_\pi(\tau_x>t).\qedhere
		\end{equation*}
\end{proof}

The next proposition shows that under the assumptions in \cref{sec:results} the tail of the hitting time $\tau_x$ starting at $\pi$ coincides---asymptotically---with the tail of $\tau_x$ starting at the ``furthest'' vertex.
%\begin{color}{blue}
	\begin{proposition}\label{le:max}
		For all $x\in \cX$ and for all $t>T$ it holds
		\begin{equation}
		\max_{y\in\cX}\P_y(\tau_x>t)\sim \P_\pi(\tau_x>t).
		\end{equation}
	\end{proposition}
	We start by proving a preliminary version of \cref{le:max}, which is expressed by the following lemma.
	\begin{lemma}\label{lemma1}
		For all $x\in\cX$ and for all $t>T$ it holds
		\begin{equation}\label{eq:lemma1}
		\max_{y\in\cX}\P_y(\tau_x>t)\lesssim \P_{\pi}(\tau_x>t-T) .
		\end{equation} 
	\end{lemma}
	\begin{proof}
		For all $x,y\in\cX$ it holds
		\begin{align}
		\P_y(\tau_x>t)=&\sum_{z\in\cX}\P_y(X_T=z;\:\tau_x>T)\P_z(\tau_x>t-T)\\
		\le&\sum_{z\in\cX}\P_y(X_T=z)\P_z(\tau_x>t-T)\\
		=&(1+o(1))\sum_{z\in\cX}\pi(z)\P_z(\tau_x>t-T)\\
		\sim&\P_\pi(\tau_x>t-T).\qedhere
		\end{align}
	\end{proof}
	Roughly, given \cref{lemma1}, the proof of \cref{le:max} follows by showing that the $-T$ term in the right hand side of \cref{eq:lemma1} does not affect the asymptotic relation. This fact is made rigorous by \cref{le:EJP} and the forthcoming \cref{corollary1}. The proof of \cref{le:EJP} is based on strong stationary times techniques (see \cite{ADuniform,DFduality,LevPer:AMS2017}) and it is inspired by the recursion in the proof of \cite[Lemma 5.4]{FMNSS}. 
		\begin{color}{black}
	Before to proceed with the proof, we need to recall some definitions and properties of strong stationary times.
	
	A randomized stopping time $\tau^\alpha_\pi$ is a \emph{Strong Stationary Time (SST)}   for the
	Markov chain $X_t$ with starting distribution $\alpha$ and stationary measure $\pi$,  if for any $t\ge 0$ and
	$y\in\cX$
	$$
	\mathbb{P}_\alpha\left(X_t=y,\tau^{\alpha}_\pi=t\right)=\pi(y)\mathbb{P}_\alpha\left(\tau^{\alpha}_\pi=t\right).
	$$
	This is equivalent to say
	\begin{equation}\label{aggiunta}
	\mathbb{P}_\alpha\left(X_t=y\big|\tau^{\alpha}_\pi\le t\right)=\pi(y).
	\end{equation}
	If $\tau^\alpha_\pi$ is a SST then
	\begin{equation}\label{minsst}
	\P_\alpha(\tau^{\alpha}_\pi>t)\ge \text{sep}(\mu^\alpha_t,\pi):= \max_{y\in\cX}\Big[1-\frac{\mu^\alpha_t(y)}{\pi(y)} \Big],\qquad \forall t\ge 0,
	\end{equation}
	and when  \cref{minsst} holds with the equal sign for every $t$, the SST is {\it minimal}. Moreover, a minimal SST always exists, see \cite[Prop. 6.14]{LevPer:AMS2017}.
	\end{color}
		\begin{lemma}\label{le:EJP}
			For any $t>0$ it holds
			\begin{equation}\label{eq:lemma-EJP}
			\frac{\P_\pi(\tau_x>t+T)}{\P_\pi(\tau_x>t)}\ge 1-o(1).
			\end{equation}
		\end{lemma}
		\begin{proof}
			
			We first prove the following inequality
			\begin{equation}\label{f1}
			\frac{\P_\pi(\tau_x>t+T)}{\P_\pi(\tau_x>t)}\ge 1-\varepsilon\,\cdot\, \frac{\P_\pi(\tau_x>t-T)}{\P_\pi(\tau_x>t)},
			\end{equation}
			with $\varepsilon=o(1)$.
			We start by rewriting
			\begin{equation}\label{eq:rewrite}
			\P_\pi(\tau_x>t+T)=\P_\pi(\tau_x>t)-\P_\pi(\tau_x\in[t,t+T]).
			\end{equation}
	Consider $\tau^z_\pi$ the minimal SST of the process started at $z$, so that the last term in \cref{eq:rewrite} can be written as
			\begin{align}
			\P_\pi(\tau_x\in[t,t+T])=&\sum_{z\in\cX}\P_\pi(\tau_x>t-T, X_{t-T}=z)\P_z(\tau_x\in[T,2T])\\
			\nonumber\le&\sum_{z\in\cX}\P_\pi(\tau_x>t-T, X_{t-T}=z)\Big[\P_z(\tau_x\in[T,2T], \tau_\pi^z\le T)+\P_z(\tau_x\le 2T, \tau_\pi^z> T)\Big]\\
			\le&\P_\pi(\tau_x>t-T)\Big[\P_\pi( \tau_x\le 2T)+\max_{z\in\cX}\P_z(\tau_\pi^z>T)\Big].
			\end{align}
			Moreover,
			\begin{equation}\label{eq:eps1}
			\P_\pi( \tau_x\le 2T)=\P_\pi\(\exists s\le 2T\text{ s.t. }X_s=x \)\le (2T+1)\pi(x)=:\varepsilon_1=o(1),
			\end{equation}
			where we used the assumption \textbf{(HP 2)}. On the other hand, thanks to \cref{le:approx}, we have
			\begin{equation}\label{eq:eps2}
			\max_{z\in\cX}\P_z(\tau_\pi^z>T)=\max_{z\in\cX} \text{sep}(\mu^z_T,\pi)\le\max_{z\in\cX} \left\| \frac{\mu^z_T}{\pi}-1\right\|_\infty=:\varepsilon_2=o(1).
			\end{equation}
			By plugging \cref{eq:eps1,eq:eps2} into \cref{eq:rewrite} we get
			\begin{equation}\label{eq:eps12}
			\P_\pi(\tau_x>t+T)\ge\P_\pi(\tau_x>t)-\P_\pi(\tau_x>t-T)(\varepsilon_1+\varepsilon_2),
			\end{equation}
			and so \cref{f1} follows with $\varepsilon:=\varepsilon_1+\varepsilon_2$.
			
			We are now going to exploit \cref{f1} to prove \cref{eq:lemma-EJP}. Consider the sequence $(y_i)_{i\ge 1}$
			\begin{equation}\label{y}
			y_i:=\frac{\P_\pi(\tau_x>(i+1)T)}{\P_\pi(\tau_x>iT)}.
			\end{equation}
			Thanks to \cref{f1} we deduce
			\begin{equation}
			\quad y_{i+1}\ge 1-\frac{\varepsilon}{y_i}.
			\end{equation}
			Being $\varepsilon<1/4$, we can define
			$$\bar\varepsilon:=\frac{1}{2}-\sqrt{\frac{1}{4}-\varepsilon}$$
			and get by induction
			\begin{equation}\label{eq:iteration}
			y_i\ge 1-\bar\varepsilon,\qquad\forall i\ge 1.
			\end{equation}
			Indeed, note that $\varepsilon=\bar\varepsilon(1-\bar\varepsilon)<\bar\varepsilon$
			$$
			y_1=\frac{\P_\pi(\tau_x>2T)}{\P_\pi(\tau_x>T)}= 1-\frac{\P_\pi(\tau_x\in [T,2T])}{\P_\pi(\tau_x>T)}\ge 1-\frac{(T+1)\pi(x)}{1-(T+1)\pi(x)}\ge 1-\frac{\varepsilon}{1-\varepsilon}
			\ge 1-\bar\varepsilon$$
			and
			$$
			y_{i+1}\ge 1-\frac{\varepsilon}{y_i}\ge 1-\frac{\varepsilon}{ 1-\bar\varepsilon}\ge 1-\bar\varepsilon.
			$$
			The result of the induction in \cref{eq:iteration} can be immediately  extended from times $iT$ to general times $t=iT+t_0$ with $t_0<T$ by noting
			that again we get 
			\begin{equation*}
			1-\frac{(T+t_0)\pi(x)}{1-t_0\pi(x)}\ge 1-\frac{\varepsilon}{1-\varepsilon}.\qedhere
			\end{equation*}
		\end{proof}	
%	Hence, in order to prove \cref{le:max} it is sufficient to show the following corollary of \cref{le:EJP}.
	\begin{corollary}\label{corollary1}
		For all $x\in\cX$ and for all $t>T$ it holds
		\begin{equation}
		\P_{\pi}(\tau_x>t-T)\sim \P_\pi(\tau_x>t) . 
		\end{equation} 
	\end{corollary}
	\begin{proof}
		Notice that it is sufficient to show that
		\begin{equation}
		\frac{\P_\pi(\tau_x>t)}{\P_\pi(\tau_x>t-T)}\ge 1-o(1),
		\end{equation}
		which follows immediately by \cref{le:EJP}.
	\end{proof}
	
	\begin{proof}[Proof of \cref{le:max}]
	It	follows immediately by \cref{lemma1,corollary1}.
	\end{proof}
	
%\end{color}

The next proposition relates the expected hitting time of $x$ starting at stationarity, with the same expectation but starting at quasi-stationarity. 

\begin{proposition}\label{pr:mean-pi}
	For all $x\in \cX$
	$$\E_\pi[\tau_x]\sim \E_{\mu^\star_x}[\tau_x]=\frac{1}{1-\lambda_x}.$$
	Hence, by \cref{pr:abd},
	$$1-\lambda_x\sim \frac{\pi(x)}{R_T(x)}.$$
\end{proposition}

In order to prove \cref{pr:mean-pi}, a key ingredient is the following lemma, which states that $1-\lambda_x$ must be much smaller than $T^{-1}$. We will later see that such a rough bound is sufficient to recover the precise first order asymptotic of $\lambda_x$ by comparing $\E_{\mu_x^\star}[\tau_x]$ to $\E_\pi[\tau_x]$.
	\begin{lemma}\label{le:lambda-small}
		For all $x\in\cX$, it holds
		\begin{equation}
		\lambda_x^T\sim 1.
		\end{equation}
	\end{lemma}
\begin{proof}
	Start by noting that
	\begin{align}
	\lambda_x^{2T}=&\P_{\mu^\star_x}(\tau_x>2T)\\
	=&\sum_{z\neq x}\P_{\mu^\star_x}\(X_T=z,\:\tau_x>T \)\P_z\(\tau_x>T \)\\
	=&\sum_{z\neq x}\[\P_{\mu^\star_x}\(X_T=z\)-\P_{\mu^\star_x}\(\:X_T=z,\:\tau_x\le T \) \]\P_z\(\tau_x>T \)\\
\text{\cref{le:approx}	}\Longrightarrow\quad	\sim&\P_\pi(\tau_x>T)-\sum_{z\neq x}\P_{\mu^\star_x}\(X_T=z,\:\tau_x\le T \)\P_z\(\tau_x>T \)\\
		\ge&\P_\pi(\tau_x>T)-\max_z \P_z\(\tau_x>T \)\P_{\mu^\star_x}\(\tau_x\le T \)\\
\text{	\cref{le:max}}\Longrightarrow\quad\sim&\P_\pi(\tau_x>T)\(1-\P_{\mu^\star_x}\(\tau_x\le T \) \)\\
		=&\P_\pi(\tau_x>T)\(1-(1-\lambda_x^T) \).
	\end{align}
Hence
	\begin{align}
	\lambda_x^{T}\ \gtrsim \P_\pi(\tau_x>T)\ge 1-(T+1)\pi(x),
	\end{align}
	so, by \textbf{(HP 2)} we can conclude that
$
	\lambda_x^T\sim 1.
$
\end{proof}
\begin{proof}[Proof of \cref{pr:mean-pi}]
	We start with the trivial bounds
	\begin{equation}
	\sum_{t=T}^{\infty}\P_{\mu^\star_x}(\tau_x>t)\le \E_{\mu^\star_x}[\tau_x]\le T+	\sum_{t=T}^{\infty}\P_{\mu^\star_x}(\tau_x>t).
	\end{equation}
	We further notice that
	\begin{align}
	\sum_{t=T}^{\infty}\P_{\mu^\star_x}(\tau_x>t)=&\sum_{z}\P_{\mu^\star_x}(X_T=z,\tau_x>T)\sum_{t=0}^{\infty}\P_z(\tau_x>t)\\
	=&\sum_{z}\[\P_{\mu^\star_x}(X_T=z)-\P_{\mu^\star_x}(X_T=z,\tau_x\le T)\]\sum_{t=0}^{\infty}\P_z(\tau_x>t)\\
	=&\sum_{z}\[\pi(z)(1+o(1))-\P_{\mu^\star_x}(X_T=z,\tau_x\le T)\]\sum_{t=0}^{\infty}\P_z(\tau_x>t)\\
	\label{last}=&(1+o(1))\E_\pi[\tau_x]-\sum_{z}\P_{\mu^\star_x}(X_T=z,\tau_x\le T)\sum_{t=0}^{\infty}\P_z(\tau_x>t).
	\end{align}
It follows immediately by \cref{last} that
\begin{equation}
	\sum_{t=T}^{\infty}\P_{\mu^\star_x}(\tau_x>t)\le (1+o(1)) \E_\pi[\tau_x].
\end{equation}
On the other hand, 
\begin{equation}
\sum_{z}\P_{\mu^\star_x}(X_T=z,\tau_x\le T)\sum_{t=0}^{\infty}\P_z(\tau_x>t)\le \P_{\mu_x^\star}(\tau_x\le T) \cdot \sum_{t=0}^{\infty}\max_z\P_z(\tau_x>t)
\end{equation}
and thanks to \cref{le:max} we get
\begin{align}
\sum_{t=0}^{\infty}\max_z\P_z(\tau_x>t)
%=&\sum_{t=0}^{\infty}\sup_z\P_z(\tau_x>t)\\
\le&T+\sum_{t=T}^{\infty}\P_\pi(\tau_x>t)=(1+o(1))\E_\pi[\tau_x].
\end{align}
At this point, the proof is complete since
\begin{equation}
\P_{\mu_x^\star}(\tau_x\le T)=1-\lambda_x^T=o(1),
\end{equation}
where the latter asymptotics follows from \cref{le:lambda-small}.
\end{proof}
We are now in shape to prove the main result.
\begin{proof}[Proof of \cref{fvtl}]
	We start by bounding each entry of the $T$-step evolution of the quasi-stationary measure. From above we have the trivial bound: for all $x,y\in\cX$
\begin{equation}\label{eq:bound-evo}
\mu_T^{\mu^\star_x}(y)\ge \lambda^T_x\mu_x^\star(y).
\end{equation}
The latter immediately implies that for all $x\in \cX$ and $t>0$ it holds
	\begin{equation}\label{eq:pr-lb}
	\P_\pi(\tau_x>t)\gtrsim\lambda_x^{t+T}\sim \lambda_x^t.
	\end{equation}
In fact, by \cref{le:approx},
\begin{equation}
\P_{\pi}(\tau_x>t)\sim \P_{\mu_T^{\mu_x^\star}}(\tau_x>t)\ge \lambda_x^T \P_{\mu_x^\star}(\tau_x>t)=\lambda_x^{t+T}.
\end{equation}
To conclude the proof, we show a matching upper bound. Component-wise, we can upper bound
\begin{align}
\mu_T^{\mu_x^\star}(y)=&\lambda_x^T\mu_x^\star(y)+(1-\lambda_x)\sum_{s=1}^T\lambda_x^s\mu_{T-s}^x(y)\\
\label{eq:ub}\le& \lambda_x^T\mu_x^\star(y)+(1-\lambda_x)\E_x[\zeta_T(y)],
\end{align}
where $\zeta_T(y)$ denotes the local time spent by the chain in the state $y$ within time $T$, i.e.
\begin{equation}
\zeta_T(y):=\sum_{s=1}^T\ind_{X_t=y}.
\end{equation}
Notice that for all $x,y\in\cX$,  holds
\begin{equation}\label{eq:blue}
\sum_{y\in\cX}\E_x[\zeta_T(y)]= T.
\end{equation}
Hence
\begin{align}
\P_\pi(\tau_x>t)\sim&\P_{\mu_T^{\mu_x^\star}}(\tau_x>t)\\
\le&\sum_{y\in\cX}\lambda_x^T\mu_x^\star(y)\P_y(\tau_x>t)+(1-\lambda_x)\sum_{y\in\cX}\E_x[\zeta_T(y)]\P_y(\tau_x>t)\\
\le&\lambda_x^{t+T}+(1-\lambda_x)T\max_{y}\P_y(\tau_x>t)\\
=& \lambda_x^{t+T}+o\(\P_\pi(\tau_x>t)\)\\
\end{align}
where in the latter asymptotic equality we used \cref{le:lambda-small,le:max}. We then conclude that for all $x\in \cX$ and $t>T$ it holds
	\begin{equation*}
	\P_\pi(\tau_x>t)\lesssim \lambda_x^t.\qedhere
	\end{equation*}
\end{proof}

\begin{color}{black}
	\section{Controlling the Doob's transform}\label{sec:doob}
We start the section by showing that the unique vector $\gamma_x$ defined by the requirements
\begin{equation}
\lambda_x\gamma_x=[P]_x\gamma_x,\qquad\pl \mu^\star_x,\gamma_x\pr =1,
\end{equation}
can be equivalently characterized by the limits
\begin{equation}\label{419}
\gamma_x(y)=\lim_{t\to\infty}\frac{\P_y(\tau_x>t)}{\lambda_x^t},\qquad\forall y\neq x.
\end{equation}
In fact, it is an immediate consequence of \cref{eq:gamma-ratio} and $|\cX|<\infty$ that for every measure $\alpha,\alpha'$ on $\cX$, defining $\gamma_x(x)=0$ and assuming $\alpha\neq\delta_x$,  holds
\begin{equation}\label{eq:defnew2}
\frac{\pl \gamma_x,\alpha'\pr}{\pl \gamma_x,\alpha\pr }=\lim_{t\to\infty}\frac{\P_{\alpha'}(\tau_x>t)}{\P_\alpha(\tau_x>t)}.
\end{equation}
Hence, choosing $\alpha=\mu_x^\star$ and $\alpha'=\delta_y$ in the latter display we get \cref{419}. Moreover, choosing $\alpha=\mu_x^\star$ and $\alpha'=\pi$ and making use of \cref{fvtl} we get indeed the claim in \cref{co:pidotgamma}.

We now aim at proving \cref{th:gamma}. We discuss the upper and the lower bound separately. In order to ease the reading, in what follows we consider the target vertex, $x$, to be fixed.
\begin{lemma}\label{le:ub-gamma}
For all $\varepsilon>0$ and $x\in\cX$ it holds
\begin{equation}
\max_{y\in\cX\setminus\{x\}}\gamma_x(y)\le 1+\varepsilon
\end{equation}
\end{lemma}
\begin{proof}
Rewrite
\begin{align}\label{420}
\max_{y\in\cX\setminus\{x\}}\P_y(\tau_x>t)\le&\max_{y\in\cX\setminus\{x\}}\P_y\(\tau_x>t;\:\tau_\pi^y\le T \)+ \max_{y\in\cX\setminus\{x\}}\P_y\(\tau_x>t;\:\tau_\pi^y>T \)\\
\label{421}\le&\P_\pi\(\tau_x>t-T \)+ \max_{y\in\cX\setminus\{x\}}\P_y\(\tau_x>t;\:\tau_\pi^y>T \).
\end{align}
We aim at showing that
\begin{equation}\label{422}
\max_{y\in\cX\setminus\{x\}}\P_y\(\tau_x>t;\:\tau_\pi^y>T \)=o\big( \P_\pi\(\tau_x>t-T \) \big).
\end{equation}
We decompose the latter by its position at time $T$, i.e.,
\begin{align}
\max_{y\in\cX\setminus\{x\}}\P_y\(\tau_x>t;\:\tau_\pi^y>T \)=&\max_{y\in\cX\setminus\{x\}}\sum_{z\in\cX\setminus\{x \}}\P_y\(\tau_x>T;\:X_T=z;\:\tau_\pi^y>T \)\:\P_z\(\tau_x>t-T \)\\
\le&\bigg(\max_{z\in\cX\setminus\{x\}}\P_z\(\tau_x>t-T \)\bigg)\cdot\bigg( \max_{y\in\cX\setminus\{x\}}\P_y(\tau_\pi^y>t)\bigg)\\
\sim&\:\P_\pi\(\tau_x>t-T \)\cdot \max_{y\in\cX}\P_y(\tau_\pi^y>t)\\
=&\:o\big(\P_\pi\(\tau_x>t-T \) \big).
\end{align}
By inserting the bounds in \cref{421,422} into \cref{419} we deduce that
\begin{align}
\gamma_x(y)=\lim_{t\to\infty}\frac{\P_y(\tau_x>t)}{\lambda_x^t}\lesssim &\lim_{t\to\infty}\frac{\P_\pi(\tau_x>t-T)}{\lambda_x^t}=1+o(1).\qedhere
\end{align}
\end{proof}

\begin{lemma}\label{le:lb-gamma}
For all $\varepsilon>0$ and $x,y\in\cX$ with $x\neq y$ it holds
	\begin{equation}
	\gamma_x(y)\ge 1-\varepsilon-\E_y[\zeta_T(x)].
	\end{equation}
\end{lemma}
\begin{proof}
	By the same argument of the proof of \cref{le:ub-gamma} it is sufficient to show that for all $\varepsilon>0$
	\begin{equation}
	\P_y(\tau_x>t)\ge (1-\varepsilon-\E_y[\zeta_T(x)])\P_\pi\(\tau_x>t\).
	\end{equation}
	Rewrite
	\begin{align}\label{428}
\P_y(\tau_x>t)\ge &\P_y\(\tau_x>t;\:\tau_\pi^y\le T\)\\
=&\sum_{s\le T}\P_y\(\tau_x>s;\:\tau_\pi^y=s\)\P_{\pi}(\tau_x>t-s)\\
\ge&\P_{\pi}(\tau_x>t)\sum_{s\le T}\P_y\(\tau_x>s;\:\tau_\pi^y=s\)\\
=&\P_{\pi}(\tau_x>t)\P_y(\tau_x>\tau_\pi^y;\tau_\pi^y\le T)
	\end{align}
	we are left with showing that
	\begin{align}
	\P_y(\tau_x>\tau_\pi^y;\tau_\pi^y\le T)\ge& \P_y(\tau_x>T)-\P_y(\tau_\pi^y>T)\\
	=&1-\P_y(\tau_x\le T)-\varepsilon\\
	=&1-\varepsilon-\sum_{s\le T}\P_y(\tau_x=s)\\
	\ge&1-\varepsilon-\sum_{s\le T}\P_y(X_s=x)\\
	=& 1-\varepsilon- \E_y[\zeta_T(x)].\qedhere
	\end{align}
\end{proof}
	
\section{A random time perspective on the FVTL}\label{sec:conditional}
\begin{color}{black}
	
	Besides the rough bounds in \cref{eq:bounds} it is possible to have a probabilistic identity that defines the tail probability of the event $\tau_x>t$ when the Markov chain starts at $\alpha$.
	In order to provide such a representation, it has been introduced in
	\cite{MS}   the notion of \emph{conditional strong quasi-stationary time} as extension of the idea of
	\emph{ strong stationary time} introduced in \cite{ADuniform} , see also  \cite{DFduality,LevPer:AMS2017}.
	In this last section, we aim at showing how the assumptions leading to the validity of the FVTL reflect on the theory of CSQST and on the mixing behavior of the Doob's transform.
	
	Consider an irreducible Markovian kernel $P$ and a state $x\in\cX$  such that $[P]_x$ is irreducible and sub-Markovian. A randomized stopping time $\tau^\alpha_\star$ is a \emph{Conditional Strong Quasi Stationary Time (CSQST)} 
	if for any $y \in \cX\setminus\{x\}$, and $t\ge 0$
	\begin{equation}\label{defCSQST'}
	\P_{\alpha}(X_t=y,\, \tau^{\alpha}_\star=t)=\mu^\star_x(y)\P_\alpha(\tau^{\alpha}_\star=t<\tau_x).
	\end{equation}
	In other words,  $\tau^\alpha_\star$ is a CSQST if for any $y \in \cX\setminus\{x\}$, and $t\ge 0$
	\begin{equation}\label{CSQST}
	\P_{\alpha} \left( X^\alpha_t = y, \tau^\alpha_\star=t \ | \ t < \tau_x \right) 
	= 
	\mu^\star_x(y) \P_{\alpha} \left(  \tau^\alpha_\star=t \ | \ t < \tau_x \right) 
	\end{equation}
	which is equivalent to
	\begin{equation}\label{CSQSTagg}
	\P_\alpha \left( X_{\tau^\alpha_\star} = y\mid \tau^\alpha_\star< \tau_x\right) 
	= 
	\mu^\star_x(y).
	\end{equation}

	By \cref{eq:bounds} we deduce that for any initial distribution $\alpha$ on $\cX\setminus\{x\}$ and
	for any CSQST  $\tau^{\alpha}_\star$   we have for any $t\ge 0$:
	$$
	\P_\alpha(\tau^{\alpha}_\star\le t<\tau_x)=\sum_{u\le t}\lambda^{t-u}\P_\alpha(\tau^{\alpha}_\star=u<\tau_x)\le  \lambda_x^t\pl\alpha,\gamma_x\pr(1-\tilde s^{\tilde\alpha}(t)).
	$$
	
	This suggests a new notion of minimality: a
	conditional strong quasi stationary
	time  $\tau^{\alpha}_\star$   is \emph{minimal } if for any $t\ge 0$
	$$
	\P_\alpha(\tau^{\alpha}_\star\le t<\tau_x)= \lambda_x^t\pl \alpha,\gamma_x\pr (1-\tilde s^{\tilde\alpha}(t)).
	$$
	The existence of minimal CSQSTs is proved in \cite{MS} where it is shown the validity of the following representation formula: for any minimal CSQST  $\tau_\star^\alpha$ and for any $t\ge 0$:
	\begin{equation}\label{eq:representation}
	\P_{\alpha}\Big(\tau_x>t\Big)=\lambda_x^t\pl\alpha,\gamma_x\pr(1-\tilde s^{\tilde\alpha}(t))+\P_{\alpha}\Big( \tau^{\alpha}_{\star,x}> t\Big),
	\end{equation}
	where
	$$\tau^{\alpha}_{\star,x}:={\tau_x\wedge \tau^{\alpha}_\star}.$$
	
	As a byproduct of the FVTL and of \cref{eq:representation} it is possible to show the following result.
	\begin{proposition}\label{pr:last}
		Under the assumptions of the FVTL there exists a minimal CSQST $\tau_{\star,x}^\pi$ such that
		\begin{equation}
		\P_\pi(\tau_{\star,x}^\pi=0)\to 1.
		\end{equation}
	\end{proposition}
	In physical terms, \cref{pr:last} confirms once again the idea that, under the assumptions of the FVTL, the stationary and the quasi-stationary distributions coincide in the thermodynamic limit. 
\end{color}
\begin{proof}
 We start by rewriting the representation formula in \cref{eq:representation} in the case $\alpha=\pi$,
\begin{equation}\label{eq:representation2}
\P_{\pi}\Big(\tau_x>t\Big)=\lambda_x^t\pl\pi,\gamma_x\pr(1-\tilde s^{\tilde\pi}(t))+\P_{\pi}\Big( \tau^{\pi}_{\star,x}> t\Big).
\end{equation}
By the FVTL in \cref{fvtl} we know that \cref{eq:representation2} implies that, uniformly in $t\ge 0$,
\begin{equation}\label{eq:representation3}
\lambda_x^t\sim\lambda_x^t\pl\pi,\gamma_x\pr(1-\tilde s^{\tilde\pi}(t))+\P_{\pi}\Big( \tau^{\pi}_{\star,x}> t\Big).
\end{equation}
Thanks to \cref{co:pidotgamma} we can simplify the latter \cref{eq:representation3} and get
\begin{equation}\label{eq:representation4}
\sup_{t\ge 0}\left|\frac{\P_{\pi}\Big( \tau^{\pi}_{\star,x}> t\Big)}{\lambda^t_x}-\tilde s^{\tilde\pi}(t)\right|=o(1).
\end{equation}
We now show that the second term in the left hand side of \cref{eq:representation4} is $o(1)$ uniformly in $t\ge 0$, which implies that the same holds for the first term. In fact, by the monotonicity of the separation distance, the estimate
\begin{equation}\label{eq:representation8}
\sup_{t\ge 0}\tilde{s}^{\tilde{\pi}}(t)=o(1),
\end{equation}
is an immediate consequence of
\begin{equation}\label{eq:representation7}
\tilde{s}^{\tilde{\pi}}(0)=o(1).
\end{equation}
In order to prove \cref{eq:representation7},
 start by noting that the stationary distribution of the Doob's transform is given by
\begin{equation}\label{eq:representation5}
\nu_x(y)=\mu_x^\star(y)\gamma_x(y),
\end{equation}
while its starting distribution is, by \cref{starting},
\begin{equation}\label{eq:representation6}
\tilde{\pi}(y)=\frac{\pi(y)\gamma_x(y)}{\pl \pi,\gamma_x\pr}\sim \pi(y)\gamma_x(y),
\end{equation}
where in the latter approximation we used again \cref{co:pidotgamma}. Hence,
\begin{equation}
\text{sep}(\nu_x,\tilde\pi)=\max_{y\in\cX\setminus\{x\}}\Big[1-\frac{\nu_x(y)}{\tilde\pi(y)} \Big]\sim \max_{y\in\cX\setminus\{x\}}\Big[1-\frac{\mu^\star_x(y)}{\pi(y)} \Big].
\end{equation}
Therefore, to prove \cref{eq:representation7},
it suffices to show that
\begin{align}\label{eq:representation9}
\max_{y\in\cX\setminus\{x\}}\Big[1-\frac{\mu^\star_x(y)}{\pi(y)} \Big]=o(1).
\end{align}
Notice that for all $y\in\cX\setminus\{x\}$ it holds
\begin{align}
\mu_x^\star(y)=&\lambda_x^{-T}\sum_{z\neq x}\mu_x^\star(z)\big([P]_x\big)^T(z,y)\\
\le & \lambda_x^{-T}\sum_{z\neq x}\mu_x^\star(z)P^T(z,y)\\
\text{\cref{le:approx}}\Longrightarrow\quad=& \lambda_x^{-T}\sum_{z\neq x}\mu_x^\star(z)\pi(y)(1+o(1))\\
\sim& \lambda_x^{-T}\pi(y)(1+o(1))\\
\text{\cref{le:lambda-small}}\Longrightarrow\quad\sim&\pi(y).
\end{align}
The latter chain of asymptotic equalities shows that \cref{eq:representation9} holds, which in turn implies \cref{eq:representation8}. Therefore, thanks to \cref{eq:representation4}, we conclude that for every minimal CSQST $\tau_{\star,x}^\pi$
\begin{equation}
\sup_{t\ge 0}\P_\pi(\tau_{\star,x}^\pi>t)=o(1).
\end{equation}
\end{proof}
\end{color}
\section*{{Acknowledgments}} 
{ 
	M.Q. was partially supported by the GNAMPA-INdAM Project 2020 ``Random walks on random games'' and PRIN 2017 project ALGADIMAR. 
}
\bibliographystyle{plain}
\bibliography{bibFVTL}

\begin{thebibliography}{10}

\bibitem{MAphd}
Mohammed Abdullah.
\newblock The cover time of random walks on graphs.
\newblock {\em PhD thesis, arXiv:1202.5569}, 2012.

\bibitem{CF3}
Mohammed Abdullah, Colin Cooper, and Alan~M. Frieze.
\newblock Cover time of a random graph with given degree sequence.
\newblock {\em Discrete Mathematics}, 312(21):3146--3163, 2012.

\bibitem{AD1}
David Aldous and Persi Diaconis.
\newblock Shuffling cards and stopping times.
\newblock {\em The American Mathematical Monthly}, 93(5):333--348, 1986.

\bibitem{AD2}
David Aldous and Persi Diaconis.
\newblock Strong uniform times and finite random walks.
\newblock {\em Advances in Applied Mathematics}, 8(1):69--97, 1987.

\bibitem{ADuniform}
David Aldous and Persi Diaconis.
\newblock Strong uniform times and finite random walks.
\newblock {\em Advances in Applied Mathematics}, 8(1):69 -- 97, 1987.

\bibitem{AF}
David Aldous and James~Allen Fill.
\newblock Reversible {M}arkov chains and random walks on graphs.
\newblock Unfinished monograph, recompiled 2014, available at
  \url{http://www.stat.berkeley.edu/\~aldous/RWG/book.html}, 2002.

\bibitem{A82}
David~J. Aldous.
\newblock Markov chains with almost exponential hitting times.
\newblock {\em Stochastic Processes and their Applications}, 13(3):305 -- 310,
  1982.

\bibitem{AB1}
David~J Aldous and Mark Brown.
\newblock Inequalities for rare events in time-reversible markov chains i.
\newblock {\em Lecture Notes-Monograph Series}, pages 1--16, 1992.

\bibitem{AB2}
David~J Aldous and Mark Brown.
\newblock Inequalities for rare events in time-reversible markov chains ii.
\newblock {\em Stochastic Processes and their Applications}, 44(1):15--25,
  1993.

\bibitem{BGM}
A~Bianchi, A~Gaudilli{\`e}re, and P~Milanesi.
\newblock On soft capacities, quasi-stationary distributions and the pathwise
  approach to metastability.
\newblock {\em Journal of Statistical Physics}, 181(3):1052--1086, 2020.

\bibitem{BG}
Alessandra Bianchi and Alexandre Gaudilli{\`e}re.
\newblock Metastable states, quasi-stationary distributions and soft measures.
\newblock {\em Stochastic Processes and their Applications}, 126(6):1622 --
  1680, 2016.

\bibitem{BdH}
Anton Bovier and Frank Den~Hollander.
\newblock {\em Metastability: a potential-theoretic approach}, volume 351.
\newblock Springer, 2016.

\bibitem{CQcover}
Pietro Caputo and Matteo Quattropani.
\newblock Stationary distribution and cover time of sparse directed
  configuration models.
\newblock {\em arXiv preprint arXiv:1909.05752}, 2019.

\bibitem{CMSM}
Pierre Collet, Servet Mart{\'\i}nez, and Jaime San~Mart{\'\i}n.
\newblock {\em Quasi-stationary distributions: Markov chains, diffusions and
  dynamical systems}.
\newblock Springer Science \& Business Media, 2012.

\bibitem{CF6}
Colin Cooper and Alan Frieze.
\newblock The cover time of sparse random graphs.
\newblock {\em Random Structures \& Algorithms}, 30(1-2):1--16, 2007.

\bibitem{CF5}
Colin Cooper and Alan Frieze.
\newblock The cover time of the preferential attachment graph.
\newblock {\em Journal of Combinatorial Theory, Series B}, 97(2):269--290,
  2007.

\bibitem{CF4}
Colin Cooper and Alan Frieze.
\newblock The cover time of the giant component of a random graph.
\newblock {\em Random Structures \& Algorithms}, 32(4):401--439, 2008.

\bibitem{CF9}
Colin Cooper, Alan Frieze, and Eyal Lubetzky.
\newblock Cover time of a random graph with a degree sequence ii: Allowing
  vertices of degree two.
\newblock {\em Random Structures \& Algorithms}, 45(4):627--674, 2014.

\bibitem{CFRmultiple}
Colin Cooper, Alan Frieze, and Tomasz Radzik.
\newblock Multiple random walks in random regular graphs.
\newblock {\em SIAM Journal on Discrete Mathematics}, 23(4):1738--1761, 2010.

\bibitem{CF8}
Colin Cooper, Alan Frieze, and Tomasz Radzik.
\newblock The cover times of random walks on random uniform hypergraphs.
\newblock {\em Theoretical Computer Science}, 509:51--69, 2013.

\bibitem{CF1}
Colin Cooper and Alan~M. Frieze.
\newblock The cover time of random regular graphs.
\newblock {\em {SIAM} J. Discrete Math.}, 18(4):728--740, 2005.

\bibitem{CF2}
Colin Cooper and Alan~M. Frieze.
\newblock Stationary distribution and cover time of random walks on random
  digraphs.
\newblock {\em J. Comb. Theory, Ser. {B}}, 102(2):329--362, 2012.

\bibitem{DS}
John~N Darroch and Eugene Seneta.
\newblock On quasi-stationary distributions in absorbing discrete-time finite
  markov chains.
\newblock {\em Journal of Applied Probability}, 2(1):88--100, 1965.

\bibitem{DFduality}
Persi Diaconis and James~Allen Fill.
\newblock Strong stationary times via a new form of duality.
\newblock {\em Ann. Probab.}, 18(4):1483--1522, 10 1990.

\bibitem{DMqs}
Persi Diaconis and Laurent Miclo.
\newblock On times to quasi-stationarity for birth and death processes.
\newblock {\em Journal of Theoretical Probability}, 22(3):558--586, 2009.

\bibitem{DMqs2}
Persi Diaconis and Laurent Miclo.
\newblock On quantitative convergence to quasi-stationarity.
\newblock In {\em Annales de la Facult{\'e} des sciences de Toulouse:
  Math{\'e}matiques}, volume~24, pages 973--1016, 2015.

\bibitem{FMNS}
Roberto Fernandez, Francesco Manzo, Francesca Nardi, and Elisabetta Scoppola.
\newblock Asymptotically exponential hitting times and metastability: a
  pathwise approach without reversibility.
\newblock {\em Electronic Journal of Probability}, 20, 2015.

\bibitem{FMNSS}
Roberto Fernandez, Francesco Manzo, Francesca~Romana Nardi, Elisabetta
  Scoppola, and Julien Sohier.
\newblock Conditioned, quasi-stationary, restricted measures and escape from
  metastable states.
\newblock {\em The Annals of Applied Probability}, 26(2):760--793, 2016.

\bibitem{keilson}
Julian Keilson.
\newblock Rarity and exponentiality.
\newblock In {\em Markov Chain Models: Rarity and Exponentiality}, pages
  130--163. Springer, 1979.

\bibitem{LevPer:AMS2017}
David~A. Levin and Yuval Peres.
\newblock {\em Markov {C}hains and {M}ixing {T}imes}.
\newblock American Mathematical Society, Providence, RI, 2017.
\newblock Second edition. With contributions by Elizabeth L. Wilmer, With a
  chapter on ``Coupling from the past'' by James G. Propp and David B. Wilson.

\bibitem{MS}
Francesco Manzo and Elisabetta Scoppola.
\newblock Exact results on the first hitting via conditional strong
  quasi-stationary times and applications to metastability.
\newblock {\em Journal of Statistical Physics}, 174(6):1239--1262, 2019.

\bibitem{Mabsorption}
Laurent Miclo.
\newblock On absorption times and dirichlet eigenvalues.
\newblock {\em ESAIM: Probability and Statistics}, 14:117--150, 2010.

\bibitem{Meigentime}
Laurent Miclo.
\newblock An absorbing eigentime identity.
\newblock {\em Markov Processes and Related Fields}, 21, 09 2014.

\bibitem{Mmetastability}
Laurent Miclo.
\newblock On metastability.
\newblock 2020.

\bibitem{OV}
Enzo Olivieri and Maria~Eul{\'a}lia Vares.
\newblock {\em Large deviations and metastability}, volume 100.
\newblock Cambridge University Press, 2005.

\bibitem{PTtree}
Jim Pitman and Wenpin Tang.
\newblock Tree formulas, mean first passage times and kemeny's constant of a
  markov chain.
\newblock {\em Bernoulli}, 24(3):1942--1972, 08 2018.

\bibitem{Pbib}
Phil~K Pollett.
\newblock Quasi-stationary distributions: a bibliography.
\newblock {\em \url{http://www. maths. uq. edu. au/~ pkp/papers/qsds/qsds.
  pdf}}, 2008.

\end{thebibliography}
 \end{document}